\RequirePackage{fix-cm}
\documentclass[smallextended]{svjour3}       
\smartqed  

\usepackage{amsmath,amsfonts,amsthm,url,color,amssymb}
\usepackage{verbatim}
\usepackage{enumerate}
\usepackage{graphicx}
\usepackage{url}
\usepackage{tikz}

\newcommand{\rmv}[1]{}

\newcommand{\ou}{\mathcal O}
\newcommand{\rad}{\mathrm{rad}}

\newcommand{\mq}{m_{\alpha, q}(x)}

\newcommand{\G}{\mathbb G}

\newcommand{\GL}{\mathrm{GL}}

\newcommand{\F}{\mathbb{F}}

\newcommand{\ord}{\mathrm{ord}}

\begin{document}
\title{Factorization of a class of composed polynomials}


\author{Lucas Reis}


\institute{L. Reis \at Departamento de Matem\'{a}tica, Universidade 
Federal de Minas Gerais,  Belo Horizonte, Brazil \\ 
\email{lucasreismat@gmail.com}\\
\emph{Present address:} School of Mathematics and Statistics, 
Carleton University, Ottawa, Canada  }

\date{\today}

\maketitle

\begin{abstract}
In this paper, we provide the degree distribution of irreducible factors of the composed polynomial $f(L(x))$ over $\F_q$, where $f(x)\in \F_q[x]$ is irreducible and $L(x)\in \F_q[x]$ is a linearized polynomial. We further provide some applications of our main result, including lower bounds for the number of irreducible factors of $f(L(x))$, constructions of high degree irreducible polynomials and the explicit factorization of $f(x^q-x)$ under certain conditions on $f(x)$.
\keywords{factorization\and finite fields\and linearized polynomials\and $\F_q$-order}
\subclass{MSC 12E20 \and MSC 11T30}
\end{abstract}

\section{Introduction}
The factorization of polynomials as well as constructions of irreducible polynomials over finite fields play important roles in modern communications. Applications include algebraic coding theory \cite{B68}, cryptography \cite{L91} and computational number theory. Many methods on the construction of irreducible polynomials \cite{KK11} and the factorization of reducible polynomials \cite{BR18} consider compositions of the form $f(g(x))$, where $f$ is an irreducible polynomial. For a generic polynomial $g\in \F_q[x]$, there is no efficient method to determine the factorization of the composition $f(g(x))$ or even just obtain the degrees of its irreducible factors. In general \cite{C69}, the factorization of $f(g(x))$ is strongly related to the factorization of $g(x)-\alpha\in \F_{q^n}[x]$, where $\alpha\in \F_ {q^n}$ is any root of $f(x)$. If $g(x)$ has some additional field structure, the factorization of $g(x)-\alpha\in \F_{q^n}[x]$ and therefore $f(g(x))\in \F_q[x]$ may be treatable. For instance, if $g=x^d$ is a monomial, $g$ has a multiplicative structure: if $\gcd(d, q)=1$ and $\alpha_0$ is a root of $x^d-\alpha$, the roots of $x^d-\alpha$ are $\gamma\alpha_0$, where $\gamma$ varies through the roots of $x^d-1$. Butler \cite{B55} obtained the following result:

\begin{theorem}\label{thm:butler}
Let $f(x)\in \F_q[x]$ be an irreducible polynomial of degree $n$ such that any of its roots has multiplicative order $e$. Let $m$ be a positive integer such that $\gcd(m, q)=1$ and $m=m_1m_2$, where $\gcd(m_1, e)=1$ and each prime factor of $m_2$ divides $e$. Then
\begin{enumerate}[(i)]
\item each root of $f(x^m)$ has multiplicative order of the form $gm_2e$, where $g|m_1$;
\item if $g|m_1$, then $f(x^m)$ has exactly $\frac{nm_2\varphi(g)}{\ord_{gm_2e}q}$ irreducible factors of degree $\ord_{gm_2e}q$ with roots of multiplicative order $gm_2e$, where $\varphi$ is the Euler Phi Function and $\ord_{b}a$ is the order of $a$ modulo $b$.
\end{enumerate}
\end{theorem}
Recently \cite{BR18}, the authors propose an efficient method to obtain the factorization of $f(x^m)$ under some special conditions on $m$ and $f(x)$. Another interesting class of polynomials is the \emph{linearized} polynomials $L=\sum_{i=0}^ma_ix^{q^i}\in \F_q[x]$. These polynomials induce $\F_q$-linear maps in any finite extension of $\F_q$. Using the additive structure of linearized polynomials, Long and Vaughan \cite{LV75} obtain an implicit description on the degree distribution of irreducible factors of $f(L(x))$ in the case when $L$ is linearized. However, the degrees of the irreducible factors are not explicitly given and the number of irreducible factors of each degree depends on the kernels of linear maps in extensions of $\F_q$; perhaps, this is due to the methods in Linear Algebra employed. The aim of this paper is to provide a far more explicit version of such result, in the sense that all the quantities depend only on elementary functions. Our approach relies on the $\F_q$-order of elements in finite fields, that corresponds to an additive analogue of the multiplicative order. In this correspondence, polynomial analogues of many number theoretic functions arise. In particular, our description yields a linearized analogue of Theorem~\ref{thm:butler}. A more detailed account in our main result provides a lower bound for the number $N$ of irreducible factors of $f(L(x))$ over $\F_q$ and, in particular, we obtain a characterization of the irreducible polynomials of the form $f(L(x))$: the irreducible polynomials arise exactly when $N=1$. We explore special cases when $N=2$ and, in particular, we obtain a method to produce high degree irreducible polynomials from primitive polynomials. We further present an efficient method to obtain the explicit factorization of $f(L(x))\in \F_q[x]$ in the case when $L=x^q-x$ and $f$ is an irreducible polynomial of degree $n$ and trace zero, where $n$ is relatively prime with $q$.

\section{Preliminaries}
In this section, we provide a background material that is used along the paper. Throughout this paper, we fix $\F_q$ the finite field with $q$ elements, where $q$ is a power of a prime $p$. We start with some basic notations: $\overline{\F}_q$ is the algebraic closure of $\F_q$ and $\ord(\alpha):=\min \{d>0\,|\, \alpha^d=1\}$ is the multiplicative order of $\alpha\in \overline{\F}_q^*$. For $\alpha\in \overline{\F}_q$, $m_{\alpha}(x)\in \F_q[x]$ is the minimal polynomial of $\alpha$ over $\F_q$ and $\deg(\alpha):=\deg(m_{\alpha})$ is the degree of $\alpha$ over $q$. It is known that, if $d=\deg(\alpha)$, then $\F_{q^d}$ is the smallest extension of $\F_q$ that contains $\alpha$ and $m_{\alpha}(x)=\prod_{i=0}^{d-1}(x-\alpha^{q^i})$: the elements $\alpha^{q^i}$ are the \emph{conjugates} of $\alpha$. In the following theorem, we summarize some basic facts on the multiplicative order of elements in finite fields.
 
\begin{theorem}\label{thm:mult}
Let $\alpha\in \overline{\F}_q^*$ be an element of multiplicative order $\ord(\alpha)=e$. The following hold:

\begin{enumerate}[(i)]
\item $\deg(\alpha)=\ord_eq$;
\item if $\beta=\alpha^s$, then $\ord(\beta)=\frac{e}{\gcd(e, s)}$.
\end{enumerate}
In addition, for any positive integer $E$ relatively prime with $q$, there exist $\varphi(E)$ elements $\alpha\in \overline{\F}_q^*$ such that $\ord(\alpha)=E$.
\end{theorem}

\subsection{Linearized polynomials and the $\F_q$-order}
For a polynomial $g\in \F_q[x]$ with $g(x)=\sum_{i=0}^{m-1}a_ix^i$, the polynomial $$L_g(x):=\sum_{i=0}^{m-1}a_ix^{q^i},$$ is the \emph{$q$-associate} of $g$. Of course, $L_g$ is always a linearized polynomial; conversely, any linearized polynomial $L\in \F_q[x]$ is the $q$-associate of some polynomial in $\F_q[x]$. In the following lemma, we show that the $q$-associates have interesting arithmetic properties.
\begin{lemma}\label{lem:linear}
Let $g, h\in \F_q[x]$. The following hold:
\begin{enumerate}[(i)]
\item $L_g(x)+L_h(x)=L_{g+h}(x)$;
\item $L_g(L_h(x))=L_{gh}(x)$.
\end{enumerate}
\end{lemma}
\begin{proof}For the proof of this result, see Section 3.4 of \cite{LN}.\end{proof}
For an element $\alpha\in \overline{\F}_q$, we set $I_{\alpha}=\{g\in \F_q[x]\,|\, L_g(\alpha)=0\}$. From Lemma~\ref{lem:linear}, $I_{\alpha}$ is an ideal of $\F_q[x]$ and, if $\alpha\in \F_{q^d}$, $L_{x^d-1}(\alpha)=\alpha^{q^d}-\alpha=0$, hence $x^d-1\in I_{\alpha}$. In particular, $I_{\alpha}$ is a nontrivial ideal of $\F_q[x]$ and so $I_{\alpha}$ is generated by a polynomial $m_{\alpha, q}(x)$, which we can suppose to be monic. Therefore, for any $\alpha\in \overline{\F}_q$ and any $g\in \F_q[x]$, we have that $L_g(\alpha)=0$ if and only if $g$ is divisible by $\mq$.

\begin{definition}
For an element $\alpha\in \overline{\F}_q$, the polynomial $\mq$ is the \emph{$\F_q$-order} of $\alpha$ over $\F_q$.
\end{definition}

\noindent For instance, the element $0$ has $\F_q$-order $m_{0, q}(x)=1$ and any element $c\in \F_q^*$ has $\F_q$-order $m_{c, q}(x)=x-1$. In general, for $\alpha\in \F_{q^d}$, $m_{\alpha, q}(x)$ divides $x^d-1$; in particular, $\gcd(m_{\alpha, q}(x), x)=1$. It is straightforward to check that the $\F_q$-order of an element $\alpha$ coincides with the $\F_q$-order of any of its conjugates $\alpha^{q^i}$. The $\F_q$-order of an element is the \emph{additive} analogue of the multiplicative order in finite fields: in this analogy, Theorem~\ref{thm:mult} can be translated to $\F_q$-order with a suitable change of functions. 

\begin{definition}
Let $f, g\in \F_q[x]$.
\begin{enumerate}[(i)]
\item the \emph{norm} of $f$ is $N(f)=q^{d}$, where $d=\deg(f)$;
\item the \emph{Euler Phi function} for polynomials over $\F_q$ is $$\Phi_q(f)=\left |\left(\frac{\F_q[x]}{\langle f\rangle}\right)^{*}\right |,$$ where $\langle f\rangle$ is the ideal generated by $f$ in $\F_q[x]$;
\item if $\gcd(f, g)=1$, $\ou(f, g):=\min\{k>0\,|\, f^k\equiv 1\pmod g\}$ is the order of $f$ modulo $g$.
\end{enumerate}
\end{definition}

The function $\Phi_q$ is multiplicative (Chinese Remainder Theorem) and $$\Phi_q(g^s)=q^{(s-1)d}(q^d-1)=N(g)^{s-1}(N(g)-1),$$ if $g$ is an irreducible polynomial of degree $d$ and $s$ is a positive integer: one may compare with $\varphi(r^s)=r^{s-1}(r-1)$ if $r$ is a prime number. It is straightforward to check that $\ou(f, g)$ divides $\Phi_q(g)$. In duality to the multiplicative order in finite fields, we obtain the following additive version of Theorem~\ref{thm:mult}.

\begin{theorem}\label{thm:add}
Let $\alpha\in \overline{\F}_q$ be an element of $\F_q$-order $\mq=h$. The following hold:
\begin{enumerate}[(i)]
\item $\deg(\alpha)=\ou(x, h)$;
\item if $\beta=L_g(\alpha)$, then $\beta$ has $\F_q$-order $m_{\beta, q}(x)=\frac{h}{\gcd(h, g)}$.
\end{enumerate}
In addition, for any polynomial $H$ relatively prime with $x$, there exist $\Phi_q(H)$ elements $\alpha\in \overline{\F}_q$ such that $\mq=H$.
\end{theorem}

\begin{proof}
\begin{enumerate}[(i)]
\item Observe that $\deg(\alpha)$ is the least positive integer $k$ such that $\alpha\in \F_{q^k}$. Also, for any positive integer $d$, we have that $\alpha\in \F_{q^d}$ if and only if $$L_{x^d-1}(\alpha)=\alpha^{q^d}-\alpha=0,$$ that is, $\mq=h$ divides $x^d-1$. Now, the result follows from definition of $\ou(x, h)$.

\item This item follows by direct calculations.
\end{enumerate}
For the proof of the last statement, see Theorem 11 of~\cite{O34}.
\end{proof}

It is well known that, for any positive integer $n$, $\sum_{d|n}\varphi(d)=n$. As follows, we also have the polynomial version of this result.

\begin{lemma}\label{lem:count}
For any polynomial $g\in \F_q[x]$ of degree $d$, the following holds:
\begin{equation}\label{eq:sum-phi}
\sum_{h|g}\Phi_q(h)=q^{d}=N(g),
\end{equation}
where $h$ is monic and polynomial division is over $\F_q$.
\end{lemma}

\begin{proof}
We observe that if $g=x^sg_0$ with $s\ge 1$ and $\gcd(g_0, x)=1$, then $$\sum_{h|g}\Phi_q(h)=\sum_{i=0}^{s}\sum_{h|g_0}\Phi_q(hx^i)=\sum_{h|g_0}\left(\Phi_q(h)\sum_{i=1}^{s}q^{i-1}(q-1)\Phi_q(h)\right)=q^s\sum_{h|g_0}\Phi_q(h),$$ where $q^s=N(x^s)$. In particular, it is sufficient to prove Eq.~\eqref{eq:sum-phi} for the case $\gcd(g, x)=1$: in this case, the formal derivative of $L_g(x)$ is a nonzero constant and so the equation $L_g(x)=0$ has exactly $\deg(L_g)=q^d$ distinct solutions in $\overline{\F}_q$. It is straightforward to see that, for $\alpha\in \overline{\F}_q$, $L_g(\alpha)=0$ if and only if $\mq$ divides $g$ and the result follows from Theorem~\ref{thm:add}.
\end{proof}

\section{The additive analogue of Theorem~\ref{thm:butler}}
So far we have provided a duality between the multiplicative order and the $\F_q$-order in finite fields. At this point, it is clear what are the objects in this additive-multiplicative correspondence. We summarize them as follows.
\begin{center}
\begin{tabular}{ l c r }
  $\mathbb Z$ & $\leftrightarrow$ & $\F_q[x]$ \\
  $q$ & $\leftrightarrow$ & $x$\\
   $|n|=n$ & $\leftrightarrow$ & $N(f)=q^{\deg(f)}$ \\
  $\ord(\alpha)$ & $\leftrightarrow$ & $\mq$\\
  $\varphi(n)$ & $\leftrightarrow$ & $\Phi_q(f)$ \\
  $\ord_ba$ & $\leftrightarrow$ & $\ou(f, g)$\\
  primes & $\leftrightarrow$ & monic irreducible polynomials\\
  monomials & $\leftrightarrow$ & linearized polynomials
\end{tabular}
\end{center}

\noindent Motivated by this correspondence, we present the linearized version of Theorem~\ref{thm:butler}.

\begin{theorem}\label{thm:main}
Let $f(x)\in \F_q[x]$ be an irreducible polynomial of degree $n$ such that any of its roots has $\F_q$-order $h$. Let $g\in \F_q[x]$ be a monic polynomial such that $\gcd(g, x)=1$ and write $g=g_1g_2$, where $\gcd(g_1, h)=1$ and each irreducible factor of $g_2$ divides $h$. If $L_g$ denotes the $q$-associate of $g$ and $\deg(g_2)=m$, then
\begin{enumerate}[(i)]
\item each root of $f(L_g(x))$ has $\F_q$-order of the form $Gg_2h$, where $G$ divides $g_1$;
\item if $G$ divides $g_1$, $f(L_g(x))$ has exactly $$\frac{nN(g_2)\Phi_q(G)}{\ou(x, Gg_2h)}=\frac{nq^m\Phi_q(G)}{\ou(x, Gg_2h)},$$ irreducible factors of degree $\ou(x, Gg_2h)$ with roots of $\F_q$-order $Gg_2h$.
\end{enumerate}
\end{theorem}

\begin{remark}
The condition $\gcd(g, x)=1$ in Theorem~\ref{thm:main} is not restrictive: if $g=x^sg_0$ with $\gcd(x, g_0)=1$, then $$f(L_g(x))=f(L_{g_0}(x))^{q^s},$$
and Theorem~\ref{thm:main} can be applied for the composition $f(L_{g_0}(x))$.
\end{remark}

\begin{example}
We consider $q=2$, $f=x^2+x+1$ and $g=x^4+x^2+x+1=(x+1)(x^3+x^2+1)$. In this case, the roots of $f$ have $\F_q$-order equals $h=x^2+1$ and then, in the notation of Theorem~\ref{thm:main}, $g=g_1g_2$, where $g_1=x^3+x^2+1$ and $g_2=x+1$. Also, $n=\deg(f)=2$ and $m=\deg(g_2)=1$,  $\ou(x, g_2h)=\ou(x, (x+1)^3)=4$ and $\ou(x, g_1g_2h)=\ou(x, (x^3+x^2+1)(x+1)^3)=28$. Also, $\Phi_q(1)=1, $ and $\Phi_q(g_1)=7$. According to Theorem~\ref{thm:main}, the polynomial
$$f(L_g(x))=f(x^{16}+x^4+x^2+x)=x^{32} + x^{16} + x^8 + x + 1,$$
has exactly $\frac{2\cdot 2^1}{4}=1$ irreducible factor of degree $4$ (with roots of $\F_q$-order $(x+1)^3$) and $\frac{2\cdot 2^1\cdot 7}{28}=1$ irreducible factor of degree $28$ (with roots of $\F_q$-order $(x^3+x^2+1)(x+1)^3$). If we compute the factorization of $f(L_g(x))$ over $\F_2$ we obtain $$x^{32} + x^{16} + x^8 + x + 1=F_1F_2,$$ where $F_1=x^4 + x + 1$ and $F_2=x^{28} + x^{25} + x^{24} + x^{22} + x^{20} + x^{19} + x^{18} + x^{17} + x^{13} + x^{12} + x^{10} + x^{8} + x^{7} + x^{6} + x^5 + x^4 + 1$.
\end{example}

Before we proceed in the proof of Theorem~\ref{thm:main}, we need some technical lemmas.

\begin{lemma}\label{lem:aux1}
For any $\alpha, \beta\in \overline{\F}_q$ and any polynomial $h\in \F_q[x]$ not divisible by $x$, the following hold:
\begin{enumerate}[(i)]
\item the polynomials $L_h(x)-\alpha$ and $L_h(x)-\beta$ are relatively prime;
\item the polynomial $L_h(x)-\alpha$ has only simple roots in $\overline{\F}_q$.
\end{enumerate}
\end{lemma}

\begin{proof}
\begin{enumerate}[(i)]
\item This item is straightforward.
\item Since $h$ is not divisible by $x$, the formal derivative of the polynomial $L_h(x)-\alpha$ is a nonzero constant and so $L_h(x)-\alpha$ cannot have repeated roots.
\end{enumerate}
\end{proof}

\begin{lemma}\label{lem:aux2}
Let $\alpha\in \overline{\F}_q$ be an element with $\F_q$-order $h=\mq$ and let $g\in \F_q[x]$ be any polynomial not divisible by $x$. If there exists a polynomial $H$ not divisible by $x$ and an element $\beta\in \overline{\F}_q$ such that $\beta$ has $\F_q$-order $H=m_{\beta, q}$ and $L_g(\beta)=\alpha$, then $h$ divides $H$ and the number of such elements $\beta$ is at most $\frac{\Phi_q(H)}{\Phi_q(h)}$.
\end{lemma}

\begin{proof}
From Theorem~\ref{thm:add}, if $L_g(\beta)=\alpha$, then $h=m_{\alpha, q}=\frac{m_{\beta, q}}{\gcd(g, m_{\beta, q})}$ and so $h$ divides $m_{\beta, q}=H$. Therefore, we have the natural group inclusion $\G_1\subseteq \G_2$, where $\G_1=\left(\frac{\F_q[x]}{\langle h\rangle }\right)^*$ and $\G_2=\left(\frac{\F_q[x]}{\langle H\rangle }\right)^*$. In particular, if $M=\Phi_q(h)=|\G_1|$ and $A_1(x), \ldots, A_M(x)\in \F_q[x]$ are the polynomials such that $\gcd(A_i, h)=1$ and $\deg(A_i)<\deg(h)$, then there exist polynomials $B_1(x), \ldots, B_M(x)\in \F_q[x]$ such that $B_i\equiv A_i\pmod h$, $\gcd(B_i, H)=1$ and $\deg(B_i)<\deg(H)$.

Let $S_h\subset \overline{\F}_q$ be the set of elements with $\F_q$-order $h$. In particular, $\alpha\in S$; we claim that $S_h=C$, where $C=\{L_{A_i}(\alpha)\,|\, 1\le i\le M\}$. For this, we observe that, since $\deg(A_i)<\deg(h)$, the differences $A_i-A_j$ with $i\ne j$ are not divisible by $h$ and so $0\ne L_{A_i-A_j}(\alpha)=L_{A_i}(\alpha)-L_{A_j}(\alpha)$. In particular, $C$ has $|M|=\Phi_q(h)=|S_h|$ distinct elements. Since $\gcd(A_i, h)=1$, it follows from Theorem~\ref{thm:add} that the $\F_q$-order of any $L_{A_i}(\alpha)$ equals $h$, hence $C\subseteq S_h$ and so $C=S_h$.

\noindent In addition, if $\gamma\in \overline{\F}_q$ is any element of $\F_q$-order $H=m_{\gamma, q}$ and $L_g(\gamma)=\alpha$, from Theorem~\ref{thm:add}, $\gamma_i:=L_{B_i}(\gamma)$ has $\F_q$-order $\frac{H}{\gcd(B_i, H)}=H$ and satisfies
$$L_g(\gamma_i)=L_{g}(L_{B_i}(\gamma))=L_{B_i}(\alpha)=L_{A_i}(\alpha).$$
\noindent In particular, for any $\theta\in S_h$, the number of elements $\beta \in \overline{\F}_q$ with $\F_q$-order equals $H$ that satisfies $L_g(\beta)=\theta$ is the same. From Theorem~\ref{thm:add}, there exist $\Phi_q(H)$ elements with $\F_q$-order equals $H$ and, since $S_h$ has $\Phi_q(h)$ elements, the result follows.
\end{proof}

\subsection{Proof of Theorem~\ref{thm:main}}

\begin{proof}
Following the notation of Theorem~\ref{thm:main}, let $\alpha\in \F_{q^n}$ be any root of $f$, hence $$f(x)=\prod_{i=0}^{n-1}(x-\alpha^{q^i}).$$ Let $\beta$ be any root of $f(L_g(x))$. In particular, $L_g(\beta)$ is a root of $f$ and, without loss of generality, suppose that $L_g(\beta)=\alpha$. 

\begin{enumerate}[(i)]
\item Set $H=m_{\beta, q}$, the $\F_q$-order of $\beta$. From Lemma~\ref{lem:aux2}, $\mq=h$ divides $H$. Additionally, since $L_{gh}(\beta)=L_{h}(L_g(\beta))=L_h(\alpha)=0$, $H$ divides $gh$. In particular, there exist divisors $h_1$ of $g_1$ and $h_2$ of $g_2$ such that $H=hh_1h_2$. Because $L_g(\beta)=\alpha$, from Theorem~\ref{thm:add}, it follows that $$h=\frac{H}{\gcd(H, g)}=\frac{hh_1h_2}{\gcd(hh_1h_2, g)},$$ and so $h_1h_2=\gcd(hh_1h_2, g)$. Since $h_1h_2$ divides $g_1g_2=g$, we have that $$\gcd(hh_1h_2, g)=h_1h_2\cdot \gcd\left(h, \frac{g}{h_1h_2}\right).$$ Therefore, $\gcd(h, \frac{g}{h_1h_2})=1$, i.e., $\gcd(h, \frac{g_1}{h_1}\cdot \frac{g_2}{h_2})=1$. Recall that $\gcd(g_1, h)=1$ and every irreducible divisor of $g_2$ divides $h$: from the previous equality, we conclude that $\frac{g_2}{h_2}=1$, i.e., $h_2=g_2$. In particular, $H=Gg_2h$, where $G=h_1$ is a divisor of $g_1$.

\item For each divisor $G$ of $g_1$, let $n(G, f)$ be the number of elements $\gamma \in \overline{\F}_q$ such that $\gamma$ has $\F_q$-order $Gg_2h$  and is a root of $f(L_g(x))$. Since $$f(L_g(x))=\prod_{i=0}^{n-1}(L_g(x)-\alpha^{q^i}),$$ from Lemma~\ref{lem:aux1}, $f(L_g(x))$ has only simple roots. In particular, from the previous item, it follows that $$\sum_{G|g_1}n(G, f)=\deg(f(L_g(x)))=nq^{\deg(g)}.$$ In addition, we observe that any root $\gamma$ of $f(L_g(x))$ satisfies $L_g(\gamma)=\alpha^{q^i}$ for some $0\le i\le n-1$ and the $\F_q$-order of $\alpha^{q^i}$ equals $m_{\alpha, q}=h$. Therefore, from Lemma~\ref{lem:aux2}, it follows that $n(G, f)\le n\cdot \frac{\Phi_q(Gg_2h)}{\Phi_q(h)}$: since every prime divisor of $g_2$ divides $h$ and $\gcd(G, h)=1$, we have $$\Phi_q(Gg_2h)=\Phi_q(G)\Phi_q(g_2h)=\Phi_q(G)N(g_2)\Phi_q(h),$$ hence 
$n(G, f)\le n\cdot N(g_2)\Phi_q(G)$. From Eq.~\eqref{eq:sum-phi}, we have $\sum_{G|g_1}\Phi_q(G)=N(g_1)$ and then
\begin{align*}nq^{\deg(g)}=\sum_{G|g_1}n(G, f)\le \sum_{G|g_1}n\cdot N(g_2)\Phi_q(G)=n\cdot N(g_2) \sum_{G|g_1}\Phi_q(G)\\=n\cdot N(g_2)N(g_1)=n\cdot N(g)=nq^{\deg(g)}.\end{align*}
Therefore,  we necessarily have equality $n(G, f)=n\cdot N(g_2)\Phi_q(G)$. In particular, we have shown that, for each divisor $G$ of $g_1$, there exist $n\cdot N(g_2)\Phi_q(G)$ roots of $f(L_g(x))$ with $\F_q$-order equals $Gg_2h$. In addition, from the previous item, this describes all the roots of $f(L_g(x))$. From item (i) of Theorem~\ref{thm:add}, an element $\beta\in \overline{\F}_q$ with $\F_q$-order $Gg_2h$ has degree $\deg(\beta)=\ou(x, Gg_2h)$: in this case, the $n\cdot N(g_2)\Phi_q(G)$ roots of $f(L_g(x))$ with $\F_q$-order $Gg_2h$ are divided in $$\frac{n\cdot N(g_2)\Phi_q(G)}{\ou(x, Gg_2h)}=\frac{nq^{\deg(g_2)}\Phi_q(G)}{\ou(x, Gg_2h)}=\frac{n q^m\Phi_q(G)}{\ou(x, Gg_2h)}$$ sets, according to their minimal polynomial over $\F_q$. In conclusion, for each divisor $G$ of $g_1$, $f(L_g(x))$ has $\frac{nq^m\Phi_q(G)}{\ou(x, Gg_2h)}$ irreducible factors of degree $\ou(x, Gg_2h)$ and this describes the factorization of $f(L_g(x))$ over $\F_q$.
\end{enumerate}
\end{proof}

\section{Applications of Theorem~\ref{thm:main}}
In this section, we provide some consequences of Theorem~\ref{thm:main}. We observe that, under the conditions of Theorem~\ref{thm:main}, the number $NI(f, g)$ of irreducible factors of $f(L_g(x))$ satisfies
\begin{equation}\label{eq:number}NI(f, g)=\sum_{G|g_1} \frac{nq^m\Phi_q(G)}{\ou(x, Gg_2h)}.\end{equation}
In particular, it is interesting to find estimates for the numbers $\ou(x, F)$, where $F$ is a polynomial not divisible by $x$. We start with the following definition.
\begin{definition}
For a polynomial $F\in \F_q[x]$, $\nu(F)$ is the greatest nonnegative integer $d$ with the property that there exists an irreducible polynomial $h\in \F_q[x]$ such that $h^d$ divides $F$. Also, $\rad(F)$ denotes the \emph{squarefree} part of $F$, i.e., $\rad(F)$ equals the product of the distinct irreducible divisors of $F$.
\end{definition}
Since finite fields are \emph{perfect fields}, $\nu(F)$ is the maximal multiplicity of a root of $F$ over $\overline{\F}_q$. Moreover, it is clear that $F$ divides $\rad(F)^{\nu(F)}$. The following lemma provides some basic facts on the numbers $\ou(x, F)$.

\begin{lemma}\label{lem:bound}
Let $F, G\in \F_q[x]$ be polynomials not divisible by $x$. The following hold:
\begin{enumerate}[(i)]
\item if $\gcd(F, G)=1$, then $\ou(x, FG)=\mathrm{lcm}(\ou(x, F), \ou(x, G))\le \ou(x, F)\cdot \ou(x, G).$ In particular, if $F$ is squarefree, $\ou(x, F)$ is not divisible by $p$.
\item $\ou(x, F)=\ou(x, \rad(F))\cdot p^r$, where $r=\lceil\log_p\nu(F)\rceil$.

\item If $\rad(F)$ divides $G$, then \begin{equation}\label{eq:ineq-1}\ou(x, FG)=\ou(x, G)\cdot p^{u},\end{equation} 
where $u=\lceil \log_p(\nu(FG))\rceil-\lceil \log_p(\nu(G))\rceil$ satisfies $$u\le \lceil \log_p(\nu(F)/\nu(G)+1)\rceil\le \left\lceil\frac{\nu(F)}{\nu(G)}\right\rceil, \nu(G)\ge 1.$$
\end{enumerate}
\end{lemma}

\begin{proof}
\begin{enumerate}[(i)]
\item The equality $\ou(x, FG)=\mathrm{lcm}(\ou(x, F), \ou(x, G))$ follows by direct calculations. If $F$ is squarefree and factors as $F=F_1\cdots F_j$, where each $F_i$ is irreducible and of degree $d_i$, $\ou(x, F)$ is the least common multiple of the numbers $\ou(x, F_i)$, for $1\le i\le j$: recall that $\ou(x, F_i)$ divides $\Phi_q(F_i)=q^{d_i}-1$, which is not divisible by $p$. In particular, $\ou(x, F)$ is not divisible by $p$.
\item Let $s=\ou(x, \rad(F))$ and $S=\ou(x, F)$. In particular, $\rad(F)$ divides $x^s-1$ and, since $\rad(F)$ is squarefree, from the previous item, it follows that $s$ is not divisible by $p$. For $r=\lceil\log_p\nu(F)\rceil$, we have $p^r>\nu(F)$, hence $(x^s-1)^{p^r}=x^{sp^r}-1$ is divisible by $\rad(F)^{\nu(F)}$. Recall that $F$ divides $\rad(F)^{\nu(F)}$, hence $F$ divides $x^{sp^r}-1$ and so $S$ divides $sp^r$. Clearly, $S$ is divisible by $s$ and then $S=sp^e$ for some nonnegative integer $e\le r$. Since $s$ is not divisible by $p$, the polynomial $x^s-1$ has no repeated irreducible factors and so $\nu(x^{sp^e}-1)=p^e$. However, $F$ divides $x^{sp^e}-1$, and then $p^e= \nu(x^{sp^e}-1)\ge \nu(F)$. Therefore, $e\ge \lceil \log_p\nu(F)\rceil=r$ and, since $e\le r$, we necessarily have $e=r$.
\item Since $\rad(F)$ divides $G$, $\rad(G)=\rad(FG)$ and so Eq.~\eqref{eq:ineq-1} follows from the previous item. We observe that $\nu(FG)\le \nu(F)+\nu(G)$ for any polynomials $F$ and $G$. In addition, $\lceil y_0-y\rceil\ge \lceil y_0\rceil-\lceil y\rceil$ for any real numbers $y_0>y>0$, and then for $\nu(G)\ge 1$, $$u=\lceil \log_p(\nu(FG))\rceil-\lceil \log_p(\nu(G))\rceil \le \lceil \log_p(\nu(F)/\nu(G)+1)\rceil.$$
To finish the proof, we observe that $\log_a(y+1)\ge y$ for any real numbers $y$ and $a>1$.
\end{enumerate}
\end{proof}

As follows, we obtain a lower bound on the number $NI(f, g)$ of irreducible factors of $f(L_g(x))$ over $\F_q$: in particular, a characterization of the irreducible polynomials of the form $f(L_g(x))$ is given. 

\begin{theorem}\label{thm:bound}
Let $f(x)\in \F_q[x]$ be an irreducible polynomial of degree $n$ such that any of its roots has $\F_q$-order $h$. Let $g\in \F_q[x]$ be a (non constant) monic polynomial such that $\gcd(g, x)=1$ and write $g=g_1g_2$, where $\gcd(g_1, h)=1$ and each irreducible factor of $g_2$ divides $h$. If $L_g$ denotes the $q$-associate of $g$ and $\deg(g_2)=m$, the number $NI(f, g)$ of irreducible factors of $f(L_g(x))$ over $\F_q$ satisfies the following:
\begin{equation}\label{eq:bound}
NI(f, g) \ge \frac{q^mW(g_1)}{p^u}\ge W(g_1)\left(\frac{q}{p}\right)^m,
\end{equation}
where $u=\lceil \log_p(\nu(g_2h))\rceil-\lceil \log_p(\nu(h))\rceil$ and $W(g_1)$ is the number of distinct monic divisors of $g_1$ over $\F_q$. In particular, $f(L_g(x))$ is irreducible over $\F_q$ if only if $q=p$ and the triple $(p, h, g)$ satisfies one of the following conditions:
\begin{enumerate}
\item $p$ is any prime and $g=H$ is a polynomial of degree one that divides $h$ but does not divide $\frac{x^n-1}{h}$;
\item $p=2$, $g=x^2+1$, $h$ is squarefree and is divisible by $x+1$.
\end{enumerate}
\end{theorem}

\begin{proof}
Let $\alpha$ be any root of $f$, hence $\deg(\alpha)=n$ and $h=\mq$. From item (i) of Theorem~\ref{thm:add}, $n=\ou(x, h)$.  From item (i) of Lemma~\ref{lem:count}, for each divisor $G$ of $g_1$, we have $\ou(x, Gg_2h)\le \ou(x, G)\cdot \ou(x, g_2h)$. We have the trivial bound $\ou(x, G)\le \Phi_q(G)$. Since $\rad(g_2)$ divides $h$, from item (iii) of Lemma~\ref{lem:bound}, we have that $$\ou(x, g_2h)=\ou(x, h)\cdot p^u=np^u.$$ Taking these estimates into Eq.~\eqref{eq:number}, we obtain the following inequality:
$$NI(f, g)=\sum_{G|g_1} \frac{nq^m\Phi_q(G)}{\ou(x, Gg_2h)}\ge \sum_{G|g_1}\frac{nq^m\Phi_q(G)}{np^u\Phi_q(G)}=\frac{q^m}{p^u}\sum_{G|g_1} 1=\frac{q^mW(g_1)}{p^u},$$
where $W(g_1)$ is the number of distinct monic divisors of $g_1$ over $\F_q$. In addition, item (iii) of Lemma~\ref{lem:bound} shows that $u\le \lceil\frac{\nu(g_2)}{\nu(h)}\rceil\le \nu(g_2)\le \deg(g_2)=m$ and so $p^u\le p^m$. This proves Inequality~\eqref{eq:bound}. 

If $f(L_g(x))$ is irreducible, then $NI(f, g)=1$: since $NI(f, g)$ is at least $W(g_1)(q/p)^m$ (which is a positive integer), if $f(L_g(x))$ is irreducible, then $W(g_1)(q/p)^m=1$ and so $q=p$ and $W(g_1)=1$, i.e., $g_1=1$ and $g_2=g$. In particular, $h$ is a non constant polynomial and so $\nu(h)\ge 1$. We now need a more restricted condition: $q^m=p^u$. Since $q=p$, it follows that $u=m$. Since $\nu(h)\ge 1$, from item (iii) of Lemma~\ref{lem:bound} we have that $u\le \lceil \log_p(\nu(g_2)/\nu(h)+1)     
\rceil=\lceil \log_p(\nu(g_2)/\nu(h)+1)     
\rceil$, hence \begin{equation}\label{eq:false}m-1=u-1<\log_p(\nu(g_2)/\nu(h)+1)\le \log_p(m+1).\end{equation} It follows by induction on $a$ and $b$ that $a^{b-1}\ge b+1$ if $a, b\ge 2$ are positive integers such that $a\ge 3$ or $b\ge 3$. In particular, Inequality~\eqref{eq:false} is false unless $m=1$ and $p$ is any prime number or $m=2$ and $p=2$. We divide into cases.

\begin{enumerate}
\item If $m=1$ and $p$ is any prime number, since $g_2=g$, it follows that $g$ equals a polynomial $H$ of degree one that divides $h$: taking account in Eq.~\eqref{eq:number} we have $NI(f, g)=\frac{np}{\ou(x, hH)}$. In particular, $f(L_g(x))$ is irreducible if and only if $\ou(x, hH)=np$. Of course, $\ou(x, hH)$ is divisible by $\ou(x, h)=n$ and so $\ou(x, hH)=np$ if and only if $\ou(x, hH)\ne n$: since, $h$ divides $x^n-1$ and $H$ is irreducible, we have that $\ou(x, hH)\ne n$ if and only if $H$ does not divide $\frac{x^n-1}{h}$.

\item If $m=2$ and $p=2$, Inequality~\eqref{eq:false} yields $$1<\log_2(\nu(g)/\nu(h)+1)\le \log_2 3,$$ hence $1<\nu(g)/\nu(h)\le 2$. Since $\nu(g)\le \deg(g)=\deg(g_2)=2$ and $\nu(h)\ge 1$, it follows that $\nu(g)=2=\deg(g)$ and $\nu(h)=1$ and so $g$ is the square of an irreducible polynomial of degree one and $h$ is squarefree. Since $\gcd(g, x)=1$ and $x, x+1$ are the only degree one irreducible polynomials over $\F_q=\F_2$, it follows that $g=(x+1)^2=x^2+1$: taking account in Eq.~\eqref{eq:number} we obtain $$NI(f, g)=\frac{4n}{\ou(x, h(x+1)^2)}.$$ In particular, $f(L_g(x))$ is irreducible if and only if $\ou(x, h(x+1)^2)=4n$. 
Let $i$ be the greatest power of $x+1$ that divides $h$; since $h$ is squarefree, $i\le 1$ and $\nu(h(x+1)^2)=2+i$. If $h_0=\mathrm{lcm}(x+1, h)$, then $\rad(h(x+1)^2)=h_0$ and $\ou(x, h_0)=n$. From item (ii) of Lemma~\ref{lem:bound}, it follows that $\ou(x, h(x+1)^2)=\ou(x, h_0)\cdot 2^r=n2^r$, where $r=\lceil \log_p(2+i)\rceil$. Therefore, $f(L_g(x))$ is irreducible if and only if $n2^r=\ou(x, h(x+1)^2)=4n$, i.e., $r=2$. The latter is equivalent to $i=1$, i.e., $h$ is divisible by $x+1$.
\end{enumerate}
\end{proof}

\begin{remark}\label{remark:zero-trace}
If we take $q=p$ a prime and $g=x-1$, condition (1) in Theorem~\ref{thm:bound} can be translated as follows: if $f(x)\in \F_p[x]$ is an irreducible polynomial of degree $n$ such that any of its roots has $\F_p$-order $h$, then $f(x^p-x)\in \F_p[x]$ is irreducible if and only if $x-1$ divides $h$ but does not divides $\frac{x^n-1}{h}$. Since $x-1$ divides $x^n-1$, this condition is equivalent to $h$ does not divide $H=\frac{x^n-1}{x-1}$. If $\alpha$ is any root of $f$, the $\F_q$-order of $\alpha$ is $h$ and so $h$ does not divide $H=\frac{x^n-1}{x-1}$ if and only if $L_H(\alpha)\ne 0$: one can see that $L_H(\alpha)=\sum_{i=0}^{n-1}\alpha^{p^i}=a_{n-1}$, where $a_{n-1}$ is the coefficient of $x^{n-1}$ in $f(x)$, commonly called the \emph{trace} of $f(x)$. 
\end{remark}

Based on the previous remark, the following corollary is straightforward.

\begin{corollary}
If $f(x)\in \F_p[x]$ is an irreducible polynomial, then $f(x^p-x)$ is irreducible if and only if the coefficient $a_{n-1}$ of $x^{n-1}$ in $f(x)$ is not zero.
\end{corollary}

The previous corollary is a particular case of a well known result; see Theorem 3.82 of \cite{LN}.

\begin{remark}
If we take $q=p=2$ and $g=x^2+1$, following the ideas in Remark~\ref{remark:zero-trace}, condition (2) in Theorem~\ref{thm:bound} can be translated as follows: if $f(x)\in \F_2[x]$ is an irreducible polynomial of degree $n$ such that any of its roots has $\F_2$-order $h$, then $f(x^4+x)\in \F_2[x]$ is irreducible if and only if $h$ is squarefree and the coefficient $a_{n-1}$ of $x^{n-1}$ in $f(x)$ is not zero. It is not hard to see that $h$ is squarefree if and only if $\ou(x, f)=n$ is not divisible by $p=2$, i.e., $n$ is odd.
\end{remark}

Based on the previous remark, we obtain the following corollary.

\begin{corollary}
If $f(x)\in \F_2[x]$ is an irreducible polynomial of degree $n$, then $f(x^4+x)\in \F_2[x]$ is irreducible if and only if the coefficient $a_{n-1}$ of $x^{n-1}$ in $f(x)$ is not zero and $n$ is odd.
\end{corollary}

We observe that the irreducible polynomials of the form $f(L_g(x))$ arising from Theorem~\ref{thm:bound} are such that every irreducible factor of $g$ divides the $\F_q$-order $h$ of the roots of $f$. In the following theorem, we consider the opposite situation. 

\begin{theorem}\label{thm:irred}
Let $f(x)\in \F_q[x]$ be an irreducible polynomial of degree $n$ such that any of its roots has $\F_q$-order $h$. Let $g\in \F_q[x]$ be a monic irreducible polynomial of degree $d\ge 1$ such that $\gcd(g, x)=\gcd(g, h)=1$ and write $e=\ou(x, g)$. Then $f(L_g(x))$ factors as one irreducible polynomial of degree $n$ and $\frac{n(q^d-1)}{\mathrm{lcm}(n, e)}$ irreducible polynomials of degree $\mathrm{lcm}(n, e)$. In particular, $f(L_g(x))$ factors into irreducible polynomials of the same degree if and only if $g$ divides $x^n-1$ and, in this case, the degree of each irreducible factor is $n$.
\end{theorem}

\begin{proof}
From item (i) of Theorem~\ref{thm:add}, $\ou(x, h)=n$. We observe that $\Phi_q(g)=q^d-1$ and, since $\gcd(g, h)=1$, item (i) of Lemma~\ref{lem:bound} yields $$\ou(x, gh)=\mathrm{lcm}(\ou(x, g), \ou(x, h))=\mathrm{lcm}(n, e).$$ From now, the degree distribution of irreducible factors of $f(L_g(x))$ follows from Theorem~\ref{thm:main}. Since the degrees of the irreducible factors of $f(L_g(x))$ are $n$ and $\mathrm{lcm}(n, e)$, the polynomial $f(L_g(x))$ factors into irreducible polynomials of the same degree $k$ if and only if $k=n$ and $\mathrm{lcm}(n, e)=n$, i.e., $e$ divides $n$. The latter is equivalent to $g$ divides $x^n-1$. 
\end{proof}

\begin{remark}\label{remark:EDF}
When a polynomial $F$ is known to factor as irreducible polynomials of the same degree over a finite field, we have an efficient probabilistic method that provides the complete factorization of $F$: in the algorithm proposed in~\cite{GS92}, if $F$ has degree $M$, the expected number of operations in $\F_q$ to obtain the factorization of $F$ over $\F_q$ is $O(M^{1.688}+M^{1+o(1)}\log q)$. 
\end{remark}

\begin{corollary}\label{cor:EDF}
If $f(x)=\sum_{i=0}^{n}a_x^i\in \F_q[x]$ is an irreducible polynomial of degree $n$ such that $a_{n-1}=0$ and $n$ is not divisible by the characteristic $p$ of $\F_q$, then $f(x^q-x)$ factors as $q$ irreducible polynomials of degree $n$.
\end{corollary}

\begin{proof}
Let $\alpha$ be any root of $f$ and let $h$ be the $\F_q$-order of $\alpha$, hence $f(x)=\prod_{i=0}^{n-1}(x-\alpha^{q^{i}})$ and so $L_H(\alpha)=\sum_{i=0}^{n-1}\alpha^{q^i}=a_{n-1}=0$. In particular, $H$ is divisible by $h$. Since $n$ is not divisible by $p$, $x^n-1=H(x-1)$ has only simple roots, hence $\gcd(H, x-1)=1$ and then $\gcd(h, x-1)=1$. From now, the result follows from Theorem~\ref{thm:irred} with $g=x-1$.

\end{proof}

If $g\ne x$ is an irreducible polynomial of degree $d$ and $\alpha\in \F_{q^d}$ is any root of $g$, we observe that $\ou(x, g)=\ord(\alpha)$: in fact, the polynomial $g$ divides $x^s-1$ if and only if $\alpha^s=1$. We have the bound $\ou(x, g)\le \Phi_q(g)=q^d-1$ and equality holds if and only if $\ord(\alpha)=q^d-1$, i.e., $\alpha$ is a generator of $\F_{q^d}^*$. In this case, $\alpha$ is a \emph{primitive} element of $\F_{q^d}$ and $g$ is commonly called a \emph{primitive polynomial}. From the previous theorem, we have the following corollary.

\begin{corollary}\label{cor:high-deg}
Let $f(x)\in \F_q[x]$ be an irreducible polynomial of degree $n$ and let $g\ne x, x-1$ be an irreducible polynomial of degree $d$ such that any of its roots has order $e=\ou(x, g)$ and $\gcd(n, e)=1$. The following hold:

\begin{enumerate}[(i)]
\item the polynomial $f(L_g(x))$ factors as one irreducible polynomial of degree $n$ and $\frac{(q^d-1)}{e}$ irreducible factors of degree $ne$;

\item if $g$ is a primitive polynomial, then $f(L_g(x))$ factors as one irreducible polynomial of degree $n$ and one irreducible polynomial of degree $n(q^d-1)$.
\end{enumerate}
\end{corollary}

\begin{proof}
Since $g\ne x-1$ and $\gcd(n, e)=1$, one can see that $g$ does not divide $h$. In particular, we are in the conditions of Theorem~\ref{thm:irred} and, since $\mathrm{lcm}(n, e)=ne$, item (i) follows from Theorem~\ref{thm:irred}. Item (ii) follows directly from item (i) with $e=q^d-1$.
\end{proof}

\subsection{Construction of high degree irreducible polynomials} Observe that item (ii) of Corollary~\ref{cor:high-deg} suggests the construction of irreducible polynomials of high degree from primitive polynomials: for instance, if $f(x)$ is an irreducible polynomial of degree $n$ and $g$ is a primitive polynomial of degree $d$ such that $\gcd(n, q^d-1)=1$ with $q^d-1>1$, then $g\ne x-1$. In particular, $f(L_g(x))$ factors as one irreducible polynomial $G_1$ of degree $n$ and one irreducible polynomial $G_2$ of degree $n(q^d-1)>n$. We have $G_1=\gcd(f(L_g(x)), x^{q^n}-x)$ and so $G_2=\frac{f(L_g(x))}{G_1}$ is an irreducible polynomial of degree $n(q^d-1)$. 

\begin{remark}
Similar ideas in the construction of irreducible polynomials of degree $n(q^d-1)$ were previously employed in \cite{KK11}, but with a different approach.
\end{remark}

\noindent In the case $q=2$, the following proposition shows that we can iterate this construction.

\begin{proposition}
Let $f\in \F_2[x]$ be an irreducible polynomial of degree $n$, let $\{d_1, \ldots, d_k\}$ be a set of positive integers pairwise relatively prime such that $d_i\ge 2$ and $\gcd(n, 2^{d_i}-1)=1$ for any $i$. In addition, let $g_1, \ldots, g_k$ be primitive polynomials such that $\deg(g_i)=d_i$. Set $f_1=f$, $n_1=n$ and for $1\le j\le k$, let $$n_{j+1}=n(2^{d_1}-1)\cdots (2^{d_{j}}-1)\; \text{and}\;f_{j+1}(x)=\frac{f_{j}(L_{g_j}(x))}{\gcd(f_{j}(L_{g_j}(x)), x^{2^{n_j}}+x)}.$$
For each $1\le j\le k+1$, $f_j(x)$ is an irreducible polynomial of degree $n_j$.
\end{proposition}

\begin{proof}
We observe that, from hypothesis, the numbers $d_i$ are pairwise relatively prime and so $\gcd(2^{d_i}-1, 2^{d_j}-1)=2^{\gcd(d_i, d_j)}-1=1$ for any $1\le i<j\le k$, i.e., the numbers $2^{d_i}-1$ are pairwise relatively prime. The fact that $f_j$ is irreducible follows after applying the argument previously given for the pair $(f, g)=(f_j, g_j)$.

\end{proof}

Of course, there is a wide variety of sets $\{d_1, \ldots, d_k\}$ that we can apply the previous corollary: for instance, one may pick $\{d_1, \ldots, d_k\}$ as a set of distinct primes and $n$ a power of two. 

\begin{example}
Consider $f_1(x)=x^4+x+1\in \F_2[x]$ an irreducible polynomial of degree $n=4$ and let $g_1=x^2+x+1$ and $g_2=x^3+x+1$ be primitive polynomials. We obtain $f_2(x)=x^{12} + x^9 + x^8 + x^6 + x^3 + x^2 + 1$ and $f_3(x)=x^{84} + x^{81} + x^{80} + x^{76} + x^{73} + x^{72} + x^{70} + x^{69} + x^{67} + x^{65} + x^{60} + x^{57} + x^{56} + x^{54} + x^{51} + x^{50} + x^{45} + x^{44} + x^{42} + x^{39} + x^{38} + x^{36} + x^{30} + x^{27} + x^{26} + x^{16} + x^{15} + x^{14} + x^{13} + x^{10} + x^8 + x^7 + x^6 + x^5 + 1$.
\end{example}

\section{The explicit factorization of $f(x^q-x)$}
From Corollary~\ref{cor:EDF}, we known that if $f$ is an irreducible polynomial of degree $n$ and trace zero, where $n$ is not divisible by the characteristic $p$ of $\F_q$, then $f(x^q-x)$ factors as $q$ irreducible polynomials of degree $n$. In this section, we provide an efficient method to obtain the explicit factorization of $f(x^q-x)$ under these conditions. We first observe that if $g(x)$ is an $n$ degree irreducible factor of $f(x^q-x)$, then for any $a\in \F_q$, $g(x+a)$ is an irreducible polynomial of degree $n$ and divides $f((x+a)^q-(x+a))=f(x^q-x)$. One may ask if $g(x+a)$ is distinct from $g(x)$. From Theorem 2.5 of~\cite{R18}, the following lemma is straightforward.

\begin{lemma}\label{lem:div-by-p}
Let $g\in \F_q[x]$ be a polynomial of degree at least $n$ and suppose that $a\in \F_q^*$ is such that $g(x+a)=g(x)$. Then $n$ is divisible by $p$. 
\end{lemma}

From the previous lemma, we obtain the following result.

\begin{corollary}\label{cor:factors}
Let $f\in \F_q[x]$ be a monic polynomial of degree $n$ such that $n$ is not divisible by $p$ and suppose that $f(x^q-x)$ factors as $q$ monic irreducible polynomials of degree $n$ over $\F_q$. If $g(x)$ is one of these irreducible factors, then $f(x)$ factors as $\prod_{a\in \F_q}g(x+a)$.
\end{corollary}
\begin{proof}
From the previous observations, for any $a\in \F_q$, $g(x+a)$ is a monic irreducible polynomial of degree $n$ that divides $f(x^q-x)$. Since there are exactly $q$ polynomial $g(x+a), a\in \F_q$ and they are all monic, it is sufficient to prove that they are all different. For this, if $g(x+a)=g(x+b)$ with $a\ne b$ elements of $\F_q$, then $g_0(x+a_0)=g_0(x)$, where $g_0(x)=g(x+b)$ is a polynomial of degree $n$ and $a_0=a-b\ne 0$. From Lemma~\ref{lem:div-by-p}, $n$ is divisible by $p$, which contradicts our hypothesis.
\end{proof}

In particular, we have shown that the knowledge of just one irreducible factor $g(x)$ of $f(x^q-x)$ is sufficient to obtain the complete factorization of such polynomial: the irreducible factors are $g(x+a)$ with $a\in \F_q$.  In the following theorem, we show how to obtain one of these irreducible factors and hence obtain the explicit factorization of $f(x^q-x)$.

\begin{theorem}\label{thm:explicit-factor}
Let $f=x^n+\sum_{i=}^{n-1}a_ix^i\in \F_q[x]$ be an irreducible polynomial of degree $n$ such that $a_{n-1}=0$ and $n$ is not divisible by the characteristic $p$ of $\F_q$. Let $\alpha$ be any root of $f$. For $$\beta=-\frac{1}{n}\sum_{i=1}^{n-1}i\alpha^{q^{n-1-i}}=-\frac{1}{n}(\alpha^{q^{n-2}}+2\alpha^{q^{n-3}}+\cdots+(n-2)\alpha^q+(n-1)\alpha),$$ the following hold.

\begin{enumerate}[(i)]
\item $\beta$ is a root of $f(x^q-x)$ and, in particular, the minimal polynomial $g(x)$ of $\beta$ has degree $n$ and satisfies $g(x)=\prod_{i=0}^{n-1}(x-\beta^{q^i})$,

\item $g(x)$ is an irreducible factor of $f(x^q-x)$ and 
$$f(x^q-x)=\prod_{a\in \F_q}g(x+a),$$
is the complete factorization of $f(x^q-x)$ over $\F_q$.
\end{enumerate}
\end{theorem}

\begin{proof}
Under our hypothesis, Corollary~\ref{cor:EDF} ensures that $f(x^q-x)$ factors as $q$ irreducible polynomials of degree $n$. In particular, we are under the conditions of Corollary~\ref{cor:factors} and so it suffices to prove that $\beta$ is a root of $f(x^q-x)$. We observe that
\begin{align*}\beta^q-\beta=-\frac{1}{n}\sum_{i=1}^{n-1}(i\alpha^{q^{n-i}}-i\alpha^{q^{n-1-i}})=-\frac{1}{n}(\alpha^{q^{n-1}}+\ldots+\alpha^q-(n-1)\alpha)\\
=-\frac{1}{n}(\alpha^{q^{n-1}}+\cdots+\alpha^q+\alpha)+\alpha=a_{n-1}+\alpha=\alpha,\end{align*}
since $0=a_{n-1}=\sum_{i=0}^{n-1}\alpha^{q^i}$. In particular, $\beta^q-\beta$ is a root of $f(x)$ and so $\beta$ is a root of $f(x^q-x)$. 

\end{proof}

\begin{remark}
Using an algorithm of Shoup (see \cite{S94}, Theorem 3.4), the minimal polynomial of $\beta$ can be obtained with $O(n^{1.688})$ operations in $\F_q$. Since $f(x^q-x)$ has degree $qn$ and factors as polynomials of the same degree, one may compare our method with the probabilistic approach of von zur Gathen and Shoup~\cite{GS92}, which gives the same factorization with $O((qn)^{1.688}+(qn)^{1+o(1)}\log q)$ operations in $\F_q$ (see Remark~\ref{remark:EDF}). In this comparison, our method is fairly better when $q$ is small, and is far more efficient if $q$ is large.
\end{remark}

\begin{example}
Let $f(x)=x^4+x-1\in \F_5[x]$ be an irreducible polynomial. In the notation of Theorem~\ref{thm:explicit-factor}, $n=4$ and $p=5$: we obtain $g(x)=x^4-x^2-x-2$ and so
$$f(x^5-x)=(x^5-x)^4+(x^5-x)-1=\prod_{i=0}^{4}[(x+i)^4-(x+i)^2-(x+i)-2],$$
or
\begin{align*}x^{20}+x^{16}+x^{12}+x^{8}+x^5+x^4 -x- 1=(x^4-x^2-x-2)\cdot (x^4+x^3+2x-1)\cdot \\
(x^4 + 2x^3 -2x^2 + x + 2)\cdot (x^4 -2x^3 -2x^2 + 2x -2)\cdot (x^4-x^3+x+2).
\end{align*}
\end{example}

In the case when $f$ has degree $2$ or has degree $3$ and $q$ is even, the minimal polynomial of $\beta$ in Theorem~\ref{thm:explicit-factor} can be explicitly computed from the coefficients of $f$ and, in particular, we obtain the factorization of special classes of polynomials over finite fields.

\begin{corollary}
Let $q$ be a power of a prime $p$. The following hold.
\begin{enumerate}[(i)]
\item if $p\ne 2$ and $a\in \F_q^*$ is a nonsquare, $f(x)=x^2-a$ is irreducible and 
$$f(x^q-x)=x^{2q}-2x^{q+1}+x^2-a=\prod_{c\in\F_q}\left(x^2+2cx+c^2-\frac{a}{4}\right).$$

\item if $p=2$ and $f(x)=x^3+ax+b$ is irreducible over $\F_q$, then 

\begin{align*}f(x^q-x)=f(x^q+x)=x^{3q}+x^{2q+1}+x^{q+2}+ax^{q}+x^3+ax+b=\\
\prod_{c\in \F_q}f(x+c)=\prod_{c\in \F_q}(x^3+cx^2+(c^2+a)x+c^3+ac+b).
\end{align*}
\end{enumerate}
\end{corollary}

\begin{proof}
We observe that, from the hypothesis in items (i) and (ii), we are under the conditions of Theorem~\ref{thm:explicit-factor} and so only the computation of the polynomial $g(x)$ is needed. 
\begin{enumerate}[(i)]
\item In this case, we have $p$ odd and $n=2$. Let $\alpha$ be a root of $f(x)=x^2-a$, hence $\alpha^2=a$. From Theorem~\ref{thm:explicit-factor}, we obtain $\beta=-\alpha/2$, $g(x)=x^2-a/4$ and the result follows.

\item In this case, $n=3$ is such that $n$ is not divisible by $p$. Let $\alpha$ be a root of $f(x)=x^3+ax+b$. From Theorem~\ref{thm:explicit-factor}, we obtain $\beta=\alpha^{q}+2\alpha=\alpha^{q}$ and so $g(x)=f(x)$ and the result follows.

\end{enumerate}

\end{proof}

\section{Conclusions}
In this paper, we provide the degree distribution of the irreducible factors of $f(L(x))$ over $\F_q$, where $f(x)\in \F_q[x]$ is irreducible and $L(x)\in \F_q[x]$ is a linearized polynomial. We further present applications of this result, including lower bounds for the number of irreducible factors of $f(L(x))$, conditions for $f(L(x))$ to be irreducible and construction of high degree irreducible polynomials. We also provide an efficient method to obtain the complete factorization of the composition $f(x^q-x)$, where $f(x)\in \F_q[x]$ is an irreducible polynomial of degree $n$ with trace zero and $\gcd(n, q)=1$.

\begin{acknowledgements}
This work was conducted during a visit to Carleton University, supported by the program CAPES-PDSE 
(process - 88881.134747/2016-01).
\end{acknowledgements}

\end{document}